\documentclass[letterpaper, 10 pt, conference]{ieeeconf}  

\IEEEoverridecommandlockouts                              

\usepackage{algorithm}
\usepackage[noend]{algorithmic}
\usepackage{graphicx}
\usepackage[usenames,dvipsnames]{xcolor}
\usepackage[font=small]{caption}
\usepackage{subfig, float, amsmath, amssymb, amsfonts, graphics, graphicx, enumitem}
\usepackage{listings,multicol}
\usepackage{lipsum}
\usepackage{hyperref}
\usepackage{dashrule}

\usepackage{bbm}		

\newtheorem{coro}{Corollary}

\newtheorem{Example}{Example}
\newtheorem{defn}{Definition}
\newtheorem{prop}{Proposition}

\newtheorem{rem}{Remark}
\newtheorem{thm}{Theorem}

\newcommand{\vx}{\textbf{x}}
\newcommand{\vy}{\textbf{y}}
\newcommand{\vz}{\textbf{z}}
\newcommand{\sx}[1]{x^{(#1)}}

\newcommand\ssucc{\succ\mkern-4mu\succ}

\newcommand{\cone}{K}
\newcommand{\decomp}{g}
\newcommand{\flowO}[1]{\Phi_{#1}}
\newcommand{\flowE}[1]{\Psi_{#1}}
\newcommand{\fplus}{f^+}
\newcommand{\fminus}{f^-}

\def\beq{\begin{equation}}
\def\eeq{\end{equation}}

\title{\LARGE \bf
On sufficient conditions for mixed monotonicity
}

\author{Liren Yang \hspace{1cm}
Oscar Mickelin \hspace{1cm}
Necmiye Ozay
\thanks{LY and NO are with the Dept. of
       Electrical Engineering and Computer Science,
       Univ. of Michigan, Ann Arbor, MI 48109
      {\tt\small yliren,necmiye@umich.edu}.  OM is with the Dept. of Mathematics, MIT, Cambridge, MA 02139 {\tt\small oscarmi@mit.edu}. This work builds on and extends an earlier technical report by LY and NO \cite{yang2017anote}.
      The work is supported in part by Ford Motor Co., NSF grants CNS-1446298 and ECCS-1553873, and DARPA grant N66001-14-1-4045. 
}
}

\begin{document}
\maketitle
\thispagestyle{empty}
\pagestyle{empty}

\begin{abstract}
Mixed monotone systems form an important class of nonlinear systems that have recently received attention in the abstraction-based control design area. Slightly different definitions exist in the literature, and it remains a challenge to verify mixed monotonicity of a system in general. In this paper, we first clarify the relation between different existing definitions of mixed monotone systems, 
and then give two sufficient conditions for mixed monotone functions defined on Euclidean space. 
These sufficient conditions are more general than the ones from the existing control literature, and they suggest that mixed monotonicity is a very generic property.  
Some discussions are provided on the computational usefulness of the proposed sufficient conditions.
\end{abstract}

\section{Introduction}


Mixed monotonicity is a property of a function that generalizes monotonicity. The latter one captures the property that the images of a function preserve the order of their pre-images, while the former one refers to the fact that a function can be decomposed into a monotonically increasing part and a monotonically decreasing part.
Apparently a monotone function is trivially a mixed monotone function with either the decreasing part or the increasing part being constant and zero.

In this paper, we study a special class of nonlinear dynamical systems called mixed monotone systems.
A notable property of such systems is that their flow maps are mixed monotone functions.
With this property, we can efficiently approximate the system's states at any time (and hence the trajectories) according to the system's initial states.
Previously, mixed monotonicity of a dynamical system has been used for qualitative system analysis, including analyzing global stability \cite{smith2008global}, \cite{chu1998mixed}, and studying convergence relation between the solutions to a parabolic system and its corresponding elliptic system \cite{lakshmikantham1998monotone}.
Recently, mixed monotone systems have attracted some attention in the area of abstraction-based controller synthesis \cite{coogan2015efficient}, \cite{yang2017fuel}.
In these works, the mixed monotonicity of a system is used for quantitative reachability analysis and abstraction computation. Moreover, unlike the earlier works focusing on qualitative analysis, these new works study mixed monotone systems defined on some compact region not necessarily invariant under the dynamics.

Despite the usefulness of mixed monotonicity in both qualitative and quantitative analysis, the definition of mixed monotone systems is not completely consistent in the literature. The authors notice that mixed monotone systems have two slightly different (but highly related) definitions in the literature \cite{coogan2016mixed}, \cite{coogan2015efficient}. Moreover, it remains unclear how to verify the mixed monotonicity of a function or a system in general. Instead, there are only some sufficient conditions \cite{chu1998mixed}, \cite{coogan2016mixed} available for checking mixed monotonicity. 
Aimed at solving these challenges, we present two main contributions of this paper. 
First, we clarify two different definitions of mixed monotone systems existing in the literature. 
Secondly, we give two sufficient conditions that can be used to verify mixed monotonicity of a function defined on  $n$-dimensional Euclidean space. 
These sufficient conditions are more general than the ones given in \cite{chu1998mixed}, \cite{coogan2016mixed}. 
By our first sufficient condition, all continuously differentiable functions with bounded partial derivatives  are mixed monotone.   
By our second sufficient condition, all functions of bounded variation are mixed monotone. 
These results suggest that mixed monotonicity holds for a large class of functions on Euclidean space.
The practical usefulness of this property, as a result, can be sometimes limited for quantitative computation. 
Hence we also provide some discussions, along the way of presenting the two sufficient conditions, with a focus on their computational usefulness.

\section{Preliminaries}\label{sec:Prelim}



Let $\mathbb{R}^n$ be $n$-dimensional Euclidean space, and let $\overline{\mathbb{R}} := \mathbb{R}\cup\{-\infty, \infty\}$ be the extended real line. By convention, we use a boldface lower-case letter, e.g. $\vx$, to denote a vector from $\mathbb{R}^n$ (or any other general vector space).
Subscript $i$ in $\vx_i$ is used to distinguish different vectors, while normal font $x_i$ is used to denote the $i^{\rm th}$ component of a vector $\vx$.

Next, we give some definitions and preliminary results related to mixed monotone functions/systems.

\begin{defn}\label{def:ProperCone} (Proper Cone \cite{boyd2004convex})
\normalfont
Let $\mathcal{X}$ be a real vector space, a set $\cone \subseteq \mathcal{X}$ is a \textit{cone} if it is closed under non-negative scaling, i.e.,
\begin{align}
\vx \in \cone, a\geq0 \Rightarrow a\vx \in \cone.
\label{eq:defCone}
\end{align}
Furthermore, a cone $\cone$ is said to be \textit{proper} if it is:
\begin{enumerate}[nolistsep]
\item[1.] convex: $\vx_1,\vx_2\in \cone, a_1,a_2\geq 0 \Rightarrow a_1\vx_1 + a_2\vx_2 \in \cone$\footnote{Note that together with Eq. \eqref{eq:defCone}, this is the same as usual convexity of a set, where $a_1\in [0,1]$ and $a_2 = 1 - a_1$.};
\item[2.] pointed: $\vx\in \cone, a <0 \Rightarrow a\vx\notin \cone$;
\item[3.] closed: $\{\vx_n\}_{n=1}^{\infty} \subseteq \cone \text{ and } \lim_{n\rightarrow \infty} \vx_n = \vx$ implies $\vx\in \cone$;
\item[4.] solid: $\cone$ has nonempty interior.
\end{enumerate}
\end{defn}

\begin{defn}\label{def:GeneralInequality} (Generalized Inequality)
\normalfont
A proper cone $\cone\subseteq \mathcal{X}$ defines a partial order on $\mathcal{X}$ in the following sense:
\begin{align}
\vx,\vy\in \mathcal{X}: \vx\succeq \vy \text{ iff } \vx-\vy \in \cone.
\label{eq:defGeq}
\end{align}
Similarly one can define $\preceq$.
\end{defn}

\begin{rem}\label{rem:OrderWellDefined}
\normalfont
The order $\succeq$ induced by a proper cone $\cone$ is indeed a partial order. First note that $\mathbf{0}\in \cone$ by letting $a = 0$ in Eq. \eqref{eq:defCone}, hence $\vx - \vx = \mathbf{0} \in \cone$, which means $\vx \succeq \vx$ and the induced order is reflexive. By convexity of $\cone$, the induced order is transitive, i.e., $\vx\succeq \vy$ and $\vy \succeq \vz$ implies that $\vx\succeq\vz$. By pointedness, the induced order is antisymmetric, i.e., $\vx \succeq \vy$ and $\vy \succeq \vx$ implies $\vx = \vy$. Moreover, if the cone $\cone$ is closed, the induced order $\succeq$ is preserved under limitation, and if $\cone$ is solid, then it allows us to define \textit{strict inequality} as $\vx \ssucc \vy \text{ iff }\vx - \vy\in \text{int}(\cone)$.
\end{rem}

\begin{defn}\label{def:Monotonicity} (Monotone Mapping)
\normalfont
Let $f: \mathcal{X}\rightarrow \mathcal{T}$ be a mapping, and let $\succeq_{\mathcal{X}}$ and $\succeq_{\mathcal{T}}$ be the partial orders induced by some cones defined on $\mathcal{X}$ and $\mathcal{T}$. The mapping $f$ is said to be \emph{monotone} if it is order preserving, that is,
\begin{align}\label{eq:OrderPreserve}
\vx, \vy \in \mathcal{X}, \vx\succeq_{\mathcal{X}}\vy \Rightarrow f(\vx)\succeq_{\mathcal{T}} f(\vy).
\end{align}
\end{defn}

\begin{defn}\label{def:MixedMonotonicity} (Mixed Monotone Mapping)
\normalfont
A mapping $f: \mathcal{X}\rightarrow \mathcal{T}$ is \emph{mixed monotone} if there exists $\decomp: \mathcal{X}\times \mathcal{X}\rightarrow \mathcal{T}$ satisfying the following:
\begin{enumerate}[nolistsep]
\item[1.] $f$ is ``embedded'' on the diagonal of $\decomp$, i.e., $\decomp(\vx,\vx) = f(\vx)$;
\item[2.] $\decomp$ is monotone increasing in terms of the first argument, i.e., $\vx_1 \succeq_{\mathcal{X}} \vx_2, \Rightarrow \decomp(\vx_1,\vy) \succeq_{\mathcal{T}} \decomp(\vx_2,\vy)$;
\item[3.] $\decomp$ is monotone decreasing in terms of the second argument, i.e., $\vy_1 \succeq_{\mathcal{X}} \vy_2 \Rightarrow \decomp(\vx,\vy_1) \preceq_{\mathcal{T}} \decomp(\vx,\vy_2)$.
\end{enumerate}
A function $\decomp$ satisfying the above conditions is called a \emph{decomposition function} of $f$.
\end{defn}

Usually, monotonicity and mixed monotonicity are defined in terms of the so called \textit{positive cone} \cite{coogan2015efficient}, \cite{angeli2003monotone}, definition of which is very similar to that of a proper cone,  except that a positive cone is not required to be closed or solid. The results that are to be presented hold for mixed monotone systems defined by a positive cone. In many important applications, however, the cones used to define the orders also turn out to be proper.

It should also be noticed that decomposition function may not be unique. To see this, consider a simple example where $f(x) = 1$. Clearly both $g(x,y)=1$ and $g(x,y) = x/y$ are decomposition functions of $f$. 
As will be discussed later in Proposition \ref{prop:OptByMM}, a decomposition function $g$ can be used to approximate the function value of $f$. We are interested in finding a $g$ that gives a tight approximation.

The following theorem allows us to approximate the values of a mixed monotone function in some region, using its decomposition function.
\begin{prop}\label{prop:OptByMM}
\normalfont
(Theorem 1 in \cite{coogan2015efficient})
Let $f: \mathcal{X}\rightarrow \mathcal{T}$ be a mapping, $\succeq_{\mathcal{X}}$ and $\succeq_{\mathcal{T}}$ be the partial orders induced by some cones defined on $\mathcal{X}$ and $\mathcal{T}$, and $X = \{\vx\in\mathcal{X}\mid \underline{\vx} \preceq_{\mathcal{X}} \vx\preceq_{\mathcal{X}} \overline{\vx}\}$. Assume $f$ is mixed monotone with decomposition function $\decomp: \mathcal{X}\times \mathcal{X}  \rightarrow \mathcal{T}$, then
\begin{align}
\decomp(\underline{\vx},\overline{\vx}) \preceq_{\mathcal{T}} f(\vx) \preceq_{\mathcal{T}} \decomp(\overline{\vx}, \underline{\vx}), \forall \vx\in X.
\label{eq:OptByMM}
\end{align}
\end{prop}

\begin{defn}\label{def:MixedMonotoneSystem} (Mixed Monotone System)
\normalfont
For simplicity, consider an autonomous system governed by a differential equation $\dot{\vx} = f(\vx)$, where $\vx\in X\subseteq \mathbb{R}^n$ is the state. Let $\flowO{t}: X \rightarrow X$ be the flow that maps the initial state at time instant $0$ to final state at time instant $t$. The system is called mixed monotone if its flow map $\flowO{t}$ is mixed monotone (i.e., satisfying Definition \ref{def:MixedMonotonicity}) for all $t$ such that $\flowO{t}$ is defined.
\end{defn}

To this point, we have all the definitions needed in this paper regarding mixed monotone functions and systems. Note that these concepts are defined for general ordered real vector spaces. In many cases, however, the space we consider is $\mathbb{R}^n$ and the partial order is induced by an orthant in $\mathbb{R}^n$. In particular, if the orthant is the positive orthant, then the induced order is simply element-wise $\leq$ in $\mathbb{R}^n$. In what follows, we only consider (mixed) monotone functions/systems in $\mathbb{R}^n$ with respect to orthant-induced orders.

\section{Main Results}\label{sec:Result}
In this section, we present the main results in this paper. 
We first clarify the relation between two different definitions of mixed monotone systems in the literature, and
then give some sufficient conditions for a function to be mixed monotone.



\subsection{On the Relation Between Two Different Definitions of Mixed Monotone Systems}

This section tries to clarify the relation between mixed monotone systems (as defined in Section \ref{sec:Prelim}) and systems with mixed monotone vector fields.
Note that the two types of systems are different conceptually: the former ones are defined to have mixed monotone flow map, while the latter ones have mixed monotone vector field.
The authors notice that both type of systems are called mixed monotone in the literature \cite{coogan2016mixed}, \cite{coogan2015efficient}, \cite{chu1998mixed}.
On the other hand, however, there is a nice result in \cite{angeli2003monotone} showing monotonicity of the vector field implies that of the flow map. Therefore, an analogous question to ask is: for a given system $\dot{\vx} = f(\vx)$, does the fact that the vector field $f$ is mixed monotone also imply the flow map $\flowO{t}$ to be mixed monotone?

To answer this question, we have the following result.
\begin{thm}\label{thm:InfCharac}
Given system $\dot{\vx} = f(\vx)$, where state $\vx \in X\subseteq \mathcal{X} = \mathbb{R}^n $ and vector field $f$ is defined on some open set $\widetilde{X}$ containing set $X$, assume that $f$ is locally Lipschitz on $\widetilde{X}$ and is mixed monotone, the system is forward complete, and the domain $X$ is positively invariant under the considered dynamics. Then, the flow map $\flowO{t}$ is mixed monotone.
\end{thm}

\begin{proof}
Denote the system by $\Sigma: \dot{\vx} = f(\vx)$. Let $f$ be a mixed monotone map and $\decomp$ be a decomposition function for $f$. We prove Theorem \ref{thm:InfCharac} by constructing a decomposition function for $\flowO{t}$ using $\decomp$.

We start by the standard trick of constructing an ``embedding system" \cite{gouze1994monotone}, i.e., consider the following system
\begin{align}
\Sigma^{\rm E}: \begin{cases}
\dot{\vx} = \decomp(\vx,\vy) \\
\dot{\vy} = \decomp(\vy,\vx)
\end{cases}
\label{eq:EmbedSys}
\end{align}
where $\decomp$ is a decomposition function of $f$. For system \eqref{eq:EmbedSys}, one can make the following observations:

\begin{itemize}
\item[(i)] the embedding system has monotone vector field, under the following order defined on $\mathcal{X}\times \mathcal{X}$:
\begin{align}
(\vx_1,\vy_1) \succeq_{\mathcal{X}\times\mathcal{X}} (\vx_2,\vy_2) \text{ iff } \vx_1 \succeq_{\mathcal{X}} \vx_2\text{ and } \vy_1 \preceq_{\mathcal{X}} \vy_2;
\label{eq:d1}
\end{align}
where $\succeq_{\mathcal{X}}$ is the element-wise order on $\mathcal{X} = \mathbb{R}^n$.
\item[(ii)] diagonal $D := \{(\vx,\vy)\in \mathcal{X}\times\mathcal{X}\mid \vx = \vy\}$ is invariant under $\flowO{t}$, the flow map of $\Sigma$;
\item[(iii)] when state $(\vx,\vy)$ stays on the diagonal $D$, we have $\dot{\vx} = f(\vx) = \dot{\vy} = f(\vy)$.
\end{itemize}
In other words, the dynamics of the system $\Sigma$ is ``embedded'' on the diagonal of that of $\Sigma^{\rm E}$.

Notice that $\Sigma^{\rm E}$ has monotone vector field (observation (i)), and its flow map is well defined on set $X\times X$ under the technical condition in the statement of Theorem  \ref{thm:InfCharac}. 
By the infinitesimal characterization of monotone systems given by \cite{angeli2003monotone}, 
 $\Sigma^{\rm E}$ has monotone flow $\flowE{t}$ under the same order in the state space $\mathcal{X}\times \mathcal{X}$, i.e.,
\begin{align}
(\vx_1,\vy_1) \succeq_{\mathcal{X}\times\mathcal{X}} (\vx_2,\vy_2) \Rightarrow & \flowE{t}(\vx_1,\vy_1) \succeq_{\mathcal{X}\times\mathcal{X}} \flowE{t}(\vx_2,\vy_2) \nonumber \\
\Rightarrow & \flowE{t}|_{\mathcal{X}}(\vx_1,\vy_1) \succeq_{\mathcal{X}} \flowE{t}|_{\mathcal{X}}(\vx_2,\vy_2),
\label{eq:d2}
\end{align}
where $\flowE{t}|_{\mathcal{X}}(\vx_1,\vy_1) $ is the projection of $\flowE{t}(\vx_1,\vy_1)$ onto $\vx$ coordinates.
Moreover, by observations (ii) and (iii), we have
\begin{align}
\flowO{t}(\vx) = \flowE{t}|_{\mathcal{X}}(\vx,\vx)
\label{eq:d3}
\end{align}
Now combining \eqref{eq:d1}, \eqref{eq:d2} and \eqref{eq:d3} leads to the fact that $\flowE{t}|_{\mathcal{X}}$ is a decomposition function of $\flowO{t}$. Hence $\flowO{t}$ is a mixed monotone map and $\Sigma$ is a mixed monotone system by definition.
\end{proof}

A few remarks are in order. The usefulness of Theorem \ref{thm:InfCharac} lies in that one can obtain a decomposition function of the vector field for a time discretization of a system from that of the associated continuous-time system.
To be specific, given a system $\dot{\vx} = f(\vx)$ satisfying the hypotheses in Theorem \ref{thm:InfCharac}, let $g$ be a decomposition function for $f$. The discrete-time system with sampling time $\Delta $ is simply governed by difference equation $\vx^{+} = \flowO{t}(\vx)$, where $\flowO{t}$ is the flow map; and $t = n\Delta$ are the sampling time instants. If one can somehow find $\flowE{t}$, the flow map of the embedding system, one automatically obtains a decomposition function for $\flowO{t}$, which is the right-hand-side of the time-discretized system equation.

In many control applications, system modeling is done in continuous-time, with the system equation derived by some governing physical principles, while there are controller design techniques developed for discrete-time models. In such cases, Theorem \ref{thm:InfCharac} can be used to leverage mixed monotonicity in the design procedure.




\subsection{A New Sufficient Condition for Mixed Monotonicity}
In this part we give a sufficient condition for a function to be mixed monotone. Particularly, we prove its sufficiency by constructing a decomposition function.
\begin{thm}\label{thm:ConstrucDF}
\normalfont
Assume $f: \mathbb{R}^n \rightarrow \mathbb{R}^m$ is differentiable, and
\begin{align}
\frac{\partial f_i}{\partial x_j}(\vx) \in (a_{ij},b_{ij}), \forall \vx\in X\subseteq \mathbb{R}^n,
\end{align}
where $a_{ij}, b_{ij}\in \overline{\mathbb{R}}$, satisfying $a_{ij}< b_{ij}$ but $(a_{ij} ,b_{ij}) \neq (-\infty,\infty)$. 
The function $f$ is mixed monotone on $X$, under element-wise order $\leq$ on  $\mathbb{R}^n$ and $\mathbb{R}^m$.
\end{thm}

\begin{proof}
We prove Theorem \ref{thm:ConstrucDF} by constructing a decomposition function for $f$, then $f$ is mixed monotone by definition.

By assumption $\frac{\partial f_i}{\partial x_j}(\vx) \in (a_{ij},b_{ij})$ for all $\vx\in X$, the interval $(a_{ij},b_{ij})$ must satisfy at least one of the following four cases:

\renewcommand\tabcolsep{3pt}
\begin{tabular}{llll}
case 1: &sign-stable   &positive & $a_{ij}\geq 0$\\
case 2: &sign-unstable &``positive'' & $a_{ij}\leq 0$, $b_{ij}\geq 0$, \\
& & & $|a_{ij}|\leq |b_{ij}|$\\
case 3: &sign-unstable &``negative'' & $a_{ij}\leq 0$,  $b_{ij}\geq 0$, \\
& & & $|a_{ij}|\geq |b_{ij}|$\\
case 4: &sign-stable   &negative & $b_{ij}\leq 0$.\\
\end{tabular}

\noindent According to the above cases, define $\decomp: \mathbb{R}^n\times \mathbb{R}^n  \rightarrow \mathbb{R}^m$ as
\begin{align}
\forall i &\in 1\dots,m: \nonumber \\
\decomp_i(\vx,\vy) &= f_i(\vz) + (\boldsymbol{\alpha}_i - \boldsymbol{\beta}_i)^T(\vx - \vy),
\label{eq:decomp_form}
\end{align}
where $\vz = [z_1,\dots, z_n]^T$, $\boldsymbol{\alpha}_i = [\alpha_{i1}, \dots, \alpha_{in}]^T$, $\boldsymbol{\beta}_i = [\beta_{i1}, \dots, \beta_{in}]^T$ are $n$-vectors defined as follows
\begin{align}
z_j = &
\begin{cases}
x_j \text{ case 1, 2}\\
y_j \text{ case 3, 4}\\
\end{cases},
\label{eq:z}\\
\alpha_{ij} = &
\begin{cases}
0 &\text{ case 1, 3, 4}\\
|a_{ij}| + \epsilon &\text{ case 2}\\
\end{cases},
\label{eq:alpha}\\
\beta_{ij} = &
\begin{cases}
0 & \text{ case 1, 2, 4}\\
-|b_{ij}| - \epsilon & \text{ case 3}\\
\end{cases},
\label{eq:beta}
\end{align}
where $\epsilon$ is a small positive number.

Next we show that $\decomp$ is a decomposition function of $f$.
\begin{enumerate}[nolistsep]
\item[1.] Obviously $\decomp(\vx,\vx) = f(\vx)$ by equations \eqref{eq:decomp_form} and \eqref{eq:z}.
\item[2.] $\vx_1 \geq \vx_2 \Rightarrow \decomp(\vx_1,\vy)\geq \decomp(\vx_2,\vy)$ because
\begin{align}
\forall i: \frac{\partial \decomp_i}{\partial x_j} = &\sum_{k = 1}^n\frac{\partial f_i}{\partial z_k}\frac{\partial z_k}{\partial x_j} + (\alpha_{ij} - \beta_{ij}) \nonumber\\
= &\frac{\partial f_i}{\partial z_j}\frac{\partial z_j}{\partial x_j} + (\alpha_{ij} - \beta_{ij}) \nonumber\\
= & \begin{cases}
\frac{\partial f_i}{\partial x_j}& \text{ case 1} \\
\frac{\partial f_i}{\partial x_j} + |a_{ij}| + \epsilon & \text{ case 2} \\
|b_{ij}| + \epsilon& \text{ case 3} \\
0& \text{ case 4} \\
\end{cases}
\nonumber \\
\geq & 0.
\end{align}
\item[3.] $\vy_1 \geq \vy_2 \Rightarrow \decomp(\vx,\vy_1)\leq \decomp(\vx,\vy_2)$ because
\begin{align}
\forall i: \frac{\partial \decomp_i}{\partial y_j} = &\sum_{k = 1}^n\frac{\partial f_i}{\partial z_k}\frac{\partial z_k}{\partial y_j} - (\alpha_{ij} - \beta_{ij}) \nonumber\\
= &\frac{\partial f_i}{\partial z_j}\frac{\partial z_j}{\partial y_j} - (\alpha_{ij} - \beta_{ij}) \nonumber \\
= &
\begin{cases}
0& \text{ case 1} \\
-|a_{ij}| - \epsilon & \text{ case 2} \\
\frac{\partial f_i}{\partial y_j} - |b_{ij}| - \epsilon& \text{ case 3} \\
\frac{\partial f_i}{\partial y_j}& \text{ case 4} \\
\end{cases}
\nonumber \\
\leq & 0.
\end{align}
\end{enumerate}
It follows from definition \ref{def:MixedMonotonicity} that $\decomp$ is a decomposition function of $f$ and hence Theorem \ref{thm:ConstrucDF} is proved.
\end{proof}

We now discuss some implications of this result. By Theorem \ref{thm:ConstrucDF}, all differentiable functions with continuous partial derivatives are mixed monotone on a compact set, because the partial derivatives are bounded on the compact set, and hence satisfy the hypothesis of Theorem \ref{thm:ConstrucDF}.

Theorem \ref{thm:ConstrucDF} is a natural extension of the result in \cite{coogan2015efficient}, which only handles the case with sign-stable partial derivatives. The idea here is to use linear terms to create additional offset to overcome the sign-unstable partial derivatives, which leads to a decomposition. 
These linear terms are chosen to be as small as possible so that the decomposition function constructed by Theorem \ref{thm:ConstrucDF} gives a tighter approximation when applying Proposition \ref{prop:OptByMM}\footnote{The proof of Theorem \ref{thm:ConstrucDF} would still go through if we combine case 2 and case 3, but we can get smaller coefficients in front of the linear term by treating these two cases separately.}.   
In the case where all the partial derivatives $\frac{\partial f_i}{\partial x_j}$ are sign-stable, the decomposition function constructed here gives a tight approximation in Proposition \ref{prop:OptByMM}, that is, the inequality in equation \eqref{eq:OptByMM} reduces to equality at some $x\in X$ \cite{coogan2015efficient}. However this is not true when there are sign-unstable partial derivatives. Thus in general the approximation given by Proposition \ref{prop:OptByMM} might be conservative when using the decomposition function constructed in Theorem \ref{thm:ConstrucDF}.
However, one can reduce such conservatism by dividing region $X$ into smaller subregions and applying the same approximation on each subregion. Then the extremum function value over region $X$ can be obtained by combining the extremum function values on those subregions. This divide-and-conquer approach, of course, requires more computational effort because one needs to approximate the ranges of sign-unstable partial derivatives on each subregion.

Note that the construction of the decomposition function requires to approximate the ranges of the sign-unstable partial derivatives.
Therefore, Theorem \ref{thm:ConstrucDF} together with Proposition \ref{prop:OptByMM} ``shift'' the difficulty of approximating the function value of $f$ into approximating its partial derivatives $\tfrac{\partial f_i}{\partial x_j}$. By doing such, the difficulty may not be reduced in general. However, in many control applications, the considered systems including thermal systems \cite{yang2017fuel} and traffic networks \cite{coogan2016mixed}, are naturally (mixed) monotone.
If one can approximate the partial derivatives of system flow once and for all and prove its (mixed) monotonicity, such properties can be used to simplify the system analysis and design techniques.



\subsection{A More General Sufficient Condition for Mixed Monotonicity}
In this part, we discuss the relation of mixed monotone functions with functions of bounded variation. This relation leads to a more general sufficient condition for mixed monotonicity. For the mixed monotone functions satisfying this condition, however, results like Proposition \ref{prop:OptByMM} may have limited practical use due to the conservatism. 
In what follows, we will consider univariate scalar-valued function for simplicity of the presentation. 

\begin{defn}\label{def:BV} (Bounded Variation) A real scalar function $f:[\underline{x},\overline{x}]\rightarrow \mathbb{R}$ is of \textit{bounded variation} if 
\begin{align}
\sup_{ P \in\mathcal{P}_{[\underline{x},\overline{x}]}} \sum_{i=0}^{|P|-1} \left\vert f\big(x^{(i+1)}\big)-f\big(x^{(i)}\big)\right\vert< + \infty,
\label{eq:TV}
\end{align}
where
\begin{enumerate}[nolistsep]
\item[1.] $P = \{\sx{1}, \sx{2},\dots, \sx{N}\}$, with $\underline{x}\ = \sx{0} < \sx{1} <\dots,<\sx{N-1}<\sx{N} = \overline{x}$, denotes a partition of interval $[\underline{x},\overline{x}]$, 
\item[2.] $|P| = N$ denotes the size of the partition,
\item[3.]  $\mathcal{P}_{[\underline{x},\overline{x}]}$ is the set of all partitions of interval $[\underline{x},\overline{x}]$.
\end{enumerate}
In particular, the value in Eq. \eqref{eq:TV} is called the \textit{total variation} of function $f$.
\end{defn}

\begin{thm}\label{thm:JordanDecomp} (Jordan Decomposition \cite{jordan1881serie}) Every function $f$ of bounded variation can be written as the sum of a monotonically increasing function $\fplus$ and a monotonically decreasing function $\fminus$, 
where the
functions $\fplus, \fminus :[\underline{x},\overline{x}]\rightarrow \mathbb{R}$ are defined by: 
\begin{align}
\fplus(x) & = \sup_{P\in\mathcal{P}_{[\underline{x},x]}} \sum_{i=0}^{|P|-1} \left\vert f\big(\sx{i+1}\big)-f\big(\sx{i}\big)\right\vert, \label{eq:fplus} \\
\fminus(x) & = f(x) - \fplus(x).
\label{eq:fminus}
\end{align}
\end{thm}

The proof of Theorem \ref{thm:JordanDecomp} is standard and can be found in \cite{royden1968real}. 
Clearly, Jordan decomposition can be used to construct decomposition functions. We state this result with the following corollary.
\begin{coro}\label{coro:BVs} A real scalar function $f:[\underline{x}, \overline{x}]\rightarrow \mathbb{R}$ is mixed monotone if it is of bounded variation. In particular, a decomposition function is $g(x,y) = \fplus(x) + \fminus(y)$.
\end{coro}

We now relate Theorem \ref{thm:ConstrucDF} and Corollary \ref{coro:BVs} for continuously differentiable functions, which are hence also of bounded variation. 
\begin{thm}\label{thm:decompBV} 
Let $f:[\underline{x},\overline{x}]\rightarrow \mathbb{R}$ be continuously differentiable. Then a decomposition function is given by $g(x,y) = f(x) + 2\int_{y}^x |f'(t)|\mathbbm{1}_{f' < 0}(t) \,\,\text{d}t = f(x) + 2 | f'(\eta_{x,y})|\cdot (x-y)$, for some $\eta_{x,y}$ in $[x,y] \cap \{ t: f'(t) < 0 \}$. Here, $\mathbbm{1}_{f' < 0}$ is the indicator function defined as $\mathbbm{1}_{f' < 0}(t) = 1$ if $f'(t) < 0$ and $\mathbbm{1}_{f' < 0}(t) = 0$ if $f'(t) \geq 0$.
\end{thm}

\begin{proof}
For a differentiable function, the total variation can be written as
\begin{align}
f^+(x) & := \!\! \int_{\underline{x}}^x \!\!| f'(t) | \,\,\text{d}t = \!\! \int_{\underline{x}}^x\!\!\! | f'(t)| - f'(t) + f'(t)\,\,\text{d}t \\
& =  \!\!\int_{\underline{x}}^x \!\!\!|f'(t)| - f'(t) \,\,\text{d}t + f(x) - f(\underline{x}) \\
&= f(x) - f(\underline{x}) + \Delta(x),
\end{align}
where
\begin{equation}
\Delta(x) = \!\! \int_{\underline{x}}^x \!\!\!| f'(t)| - f'(t) \,\,\text{d}t = 2\int_{\underline{x}}^x  \!\!\! | f'(t)| \mathbbm{1}_{f' < 0}(t) \,\,\text{d}t.
\end{equation}
Clearly $f^+(x)$ is monotonically increasing by definition.  We also have
\begin{align}
f^-(x) & : = f(x) - f^+(x) = f(\underline{x}) - \Delta(x)
\end{align}
is monotonically decreasing. This gives 
\begin{align}
g(x,y) &= f^+(x) + f^-(y) \\
&= f(x) + \Delta(x) - \Delta(y) \\
&= f(x) + 2\int_{y}^x \!\!\! |f'(t)| \mathbbm{1}_{f' < 0}(t) \,\,\text{d}t,
\end{align}
which proves the first equality in the statement. Lastly, since we assumed $f$ to be continuously differentiable, the second equality follows by applying the mean value theorem to the integral in the first equality.
\end{proof}

The second equality in Theorem \ref{thm:decompBV} is of a form similar to the decomposition function in Theorem \ref{thm:ConstrucDF} and the constant in front of the linear term can be larger or smaller than the corresponding constant in Theorem \ref{thm:ConstrucDF}, depending on the function $f $ and the points $x,y$, as shown in the following two examples.

\begin{Example}
For $f: [\underline{x},\overline{x}]\rightarrow \mathbb{R}$ defined by $f(x) = -x$, the result in Theorem \ref{thm:decompBV} gives the decomposition function $g(x,y) = x - 2y$, whereas Theorem \ref{thm:ConstrucDF} results in $g(x,y) = -y$. When these decomposition functions are inserted into Proposition \ref{prop:OptByMM}, they result in the bounds
\begin{align}
\underline{x} - 2\overline{x} \leq &-x \leq \overline{x} - 2 \underline{x} \\
-\overline{x} \leq &-x \leq -\underline{x},
\end{align}
respectively. This is therefore an example of when Theorem \ref{thm:ConstrucDF} produces a tighter bound than Theorem \ref{thm:decompBV}.
\end{Example}

\begin{Example}Let $f: [-1,1]\rightarrow \mathbb{R}$ be defined by $f(x) = x^2$. The decomposition function given by Theorem \ref{thm:decompBV} is $g(x,y) = x^2 - 2(\min\{x,0\})^2 + 2(\min\{y,0\})^2$, which leads to $-1 \leq x^2 \leq 3$. One the other hand, Theorem \ref{thm:ConstrucDF} gives $g(x,y) = x^2 + 2x - 2y$, which leads to $-3 \leq x^2 \leq 5$.
In this example,  Theorem  \ref{thm:decompBV} produces a tighter bound than Theorem  \ref{thm:ConstrucDF}.
\end{Example}

Next, several remarks are given in regard to the above results.  
First, all the results in this subsection so far are developed for univariate scalar functions. 
To extend Corollary \ref{coro:BVs} to multivariate, vector-valued functions, one needs a notion of bounded variation for multivariate functions. There are several different ways of defining the total variation of a multivariate function.
Under the definition of total variation in \cite{lenze1990constructive}, a Jordan decomposition can be found for multivariate, scalar-valued functions of bounded variation. 
A decomposition function for multivariate vector-valued functions of bounded variation can then be constructed element-wise whenever the order on the image space is induced by the positive orthant. The bounded variation argument can be pushed to include also functions with unbounded domains as shown in the Appendix.

Secondly, note that the converse of Theorem \ref{thm:JordanDecomp} is also true, i.e., every function having a Jordan decomposition
must be of bounded variation. 
This is not saying that every univariate scalar mixed monotone function must be of bounded variation.  
The reason is because there is a loss of generality in requiring the decomposition function to have a specific form, i.e., the \textit{sum} of a increasing function and a decreasing function.
 
In the context of dynamical systems, these results suggest that when the vector field $f$ is of bounded variation\footnote{Assuming $f$ also satisfies conditions on existence and uniqueness of (Carath\'eodory) solutions of the corresponding differential equation (see, e.g., \cite{cortes2008discontinuous} for such conditions) so that the flow map is uniquely defined.}, the system is mixed monotone. Given that functions that are not of bounded variation are rare, this indicates that mixed monotonicity is a quite generic property. 

Finally, the usefulness of this theoretical result can be sometimes limited when applied for computation (e.g., of reachable sets).
In order to approximate the function values using Proposition \ref{prop:OptByMM}, it requires that the decomposition function can be evaluated easily. 
However, the decomposition function $g(x,y) = \fplus(x) + \fminus(y)$ is hard to compute in general. 
Another drawback is that the approximation given by evaluating this decomposition function can be conservative. 
In fact,  the obtained upper and lower bounds using Proposition \ref{prop:OptByMM} is the function value at some point plus the upper and lower bound of the total variation \cite{royden1968real}. 
These considerations to some extent suggest that the bounded variation based sufficient condition for mixed monotonicity may be too general to be useful for computation. As Mac Lane pointed out, ``good general theory does not search for the maximum generality, but for the right generality.''


\section{Conclusion}

In this paper we studied mixed monotone functions and systems. The relation between different definitions of mixed monotone systems in the literature were clarified, and two new sufficient conditions for mixed monotonicity were derived. 
Our results suggest that mixed monotonicity is a relatively generic property. 
While the new sufficient conditions generalize an earlier sufficient condition based on sign-stability of partial derivatives of the vector field,  
the approximation technique by decomposition function can be conservative when applied to the systems satisfying the new conditions. 
Hence, finding better cones and better decompositions that would lead to tighter approximations for a mixed monotone function is still of interest.

\section*{Appendix}\label{sec:app}

The following result allows us to extend the bounded variation based arguments to functions with unbounded domains.

\begin{prop}
A function $f: \mathbb{R} \rightarrow \mathbb{R}^m$ is mixed monotone with respect to the order induced by the positive orthant if $f\big|_I$ is of bounded variation for any closed interval $I\subseteq \mathbb{R}$, where $f\big|_I$ denotes the restriction of $f$ to the interval $I$.
\end{prop}
\begin{proof}
Note that $f$ is mixed monotone with respect to the order induced by the positive orthant of $\mathbb{R}^m$ if and only if each coordinate $f_i$ is, so we can without loss of generality assume that $m=1$. Moreover, adding or subtracting a constant to $f$ will not alter whether or not it is mixed monotone. We therefore further assume that $f(0) = 0$.

Assume now that $f\big|_I$ is of bounded variation for any closed interval $I$. We construct a decomposition function $g(x,y) = g_1(x) + g_2(y)$ by defining $g_1, g_2$ separately on the two intervals $(-\infty, 0]$ and $[0, \infty)$, respectively.

For $x,y \geq 0$, define $g_1(x)$ to be the total variation of $f$ on the interval $[0,x]$, and $g_2(y) = f(y) - g_1(y)$. By Theorem \ref{thm:JordanDecomp}, $g_1(x)$ is monotonically increasing for $x \geq 0$ and $g_2(y)$ is monotonically decreasing for $y \geq 0$. We also have $g_1(0) = 0$, $g_2(0) = f(0) = 0$.

Next, for $x,y \leq 0$, define $g_2(y)$ to be the total variation of $f$ on the interval $[y,0]$, and $g_1(x) = f(x) - g_2(x)$. This implies that $g_2(y)$ is monotonically decreasing for $y \leq 0$ and $g_1(x)$ is monotonically increasing for $x \leq 0$. We also have $g_2(0) = 0$, $g_1(0) = f(0) = 0$, so the definitions of $g_1, g_2$ on the two intervals coincide at the intersection point $x, y =0$. This implies that, for any choice of pairs $x_-, y_- \leq 0$, and $x_+, y_+ \geq 0$, we have
\begin{align}
g_1(x_-) &\leq 0 \leq g_1(x_+), \\
g_2(y_-) &\geq 0 \geq g_2(y_+),
\end{align}
i.e. $g_1(x)$ is monotonically increasing for $x$ in all of $\mathbb{R}$ and $g_2(y)$ is monotonically decreasing for $y$ in all of $\mathbb{R}$. By construction, we also have $g(x,x) = g_1(x) + g_2(x) = f(x)$ for $x$ in either of the two intervals $(-\infty, 0]$ and $[0, \infty)$. It therefore follows that $g(x,y) = g_1(x) + g_2(y)$ is a decomposition function of $f$, so $f$ is mixed monotone on $\mathbb{R}$.
\end{proof}

Note that this implies that e.g. $f(x) = x\text{sin}(x)$ is mixed monotone on $\mathbb{R}$, even though its derivative becomes unbounded as $|x| \rightarrow \infty$. This is a situation that is not covered by the assumptions of Theorem \ref{thm:ConstrucDF}. 

\bibliographystyle{abbrv}
\bibliography{IEEEabrv,switching,monotone}

\end{document}